\documentclass[a4paper, 10pt]{article}  

\usepackage[a4paper, total={6.5in, 8in}]{geometry}

\usepackage[utf8]{inputenc}
\usepackage[T1]{fontenc}
\usepackage[english]{babel}

\title{\LARGE \bf
A Robust Server-Effort Policy for Fluid Processing Networks
}
\author{
    Harold~Ship\thanks{Harold Ship is with the IBM Research - Haifa Lab,Haifa University Campus, Mount Carmel, Haifa 3490002, Israel and a PhD student at the University of Haifa. (e-mail: harold@il.ibm.com).},
    Evgeny~Shindin\thanks{Evgeny Shindin is with the IBM Research - Haifa Lab, Haifa University Campus, Mount Carmel, Haifa 3490002, Israel. (e-mail: evgensh@il.ibm.com).},
	Odellia~Boni\thanks{Odellia Boni is with the IBM Research - Haifa Lab, Haifa University Campus, Mount Carmel, Haifa 3490002, Israel. (e-mail: odelliab@il.ibm.com).},
    Itai~Dattner\thanks{Itai Dattner is a Senior Lecturer at the University of Haifa Department of Statistics, University of Haifa, 199 Abba Khoushy, Haifa 3498838, Israel. (e-mail: idattner@stat.haifa.ac.il).}
    }

\usepackage[fleqn]{amsmath}
\usepackage{amsfonts} 
\usepackage{mathtools} 
\usepackage{fixmath} 
\usepackage{relsize} 
\usepackage{bm} 

\usepackage{amsthm} 

\usepackage{hyperref}

\usepackage{cleveref}

\usepackage[acronym,toc,nonumberlist]{glossaries}

\usepackage{multirow}

\usepackage{paralist}

\usepackage{tikz}
\usetikzlibrary{positioning}
\usetikzlibrary{arrows}
\usetikzlibrary{shapes.geometric}
\usetikzlibrary{fit}

\usepackage{graphicx}

\usepackage{caption}
\usepackage{subcaption}

\usepackage{booktabs}

\usepackage{algorithm}
\usepackage{algpseudocode}

\makeglossaries
\newacronym{ode}{ODE}{ordinary differential equation}
\newacronym{lp}{LP}{linear programming}
\newacronym{clp}{CLP}{continuous linear programming}
\newacronym{sclp}{SCLP}{separated continuous linear programming}
\newacronym{mcqn}{MCQN}{multi-class queueing network}
\newacronym{reqn}{REQN}{re-entrant queueing network}
\newacronym{ro}{RO}{robust optimization}
\newacronym{so}{SO}{stochastic optimization}
\newacronym{iot}{IoT}{internet of things}

\newglossaryentry{modelA}
{
    name={processing rates model},
    description={Bertsimas' state of the art}
}
\newglossaryentry{modelAabbr}
{
    name={proc-rates},
    description={Bertsimas' state of the art}
}
\newglossaryentry{modelB}
{
    name={server-effort model},
    description={our new proposed}
}
\newglossaryentry{modelBabbr}
{
    name={serv-eff},
    description={our new proposed}
}

\newcommand*{\transpose}{\top}%
\newcommand*{\dd}{\ensuremath{\mathop{}\!\mathrm{d}}}%
\renewcommand*{\L}{\ensuremath{\mathop{}\!\mathbb{L}}}%
\newcommand*{\R}{\ensuremath{\mathop{}\!\mathbb{R}}}%
\newcommand*{\RR}[1]{\ensuremath{\mathop{}\!\mathbb{R}^{#1}}}%
\newcommand*{\uncertainset}[1]{\ensuremath{\mathcal{#1}}}%
\newcommand{\Tt}{^{{\mbox{\tiny \bf \sf T}}}}
\newcommand*{\oo}[1]{{\overline{#1}}}%
\newcommand{\ot}[1]{\ensuremath{\tilde{#1}}}%

\usepackage{color}

\newtheorem{thm}{Theorem}

\crefname{thm}{theorem}{theorems}

\newtheorem{prop}[thm]{Proposition}

\crefname{prop}{proposition}{propositions}

\begin{document}

\maketitle

\thispagestyle{empty}
\pagestyle{empty}

\begin{abstract}
Multi-Class Processing Networks describe a set of servers that perform multiple classes of jobs on different items. A useful and tractable way to find an optimal control for such a network is to approximate it by a fluid model, resulting in a Separated Continuous Linear Programming (SCLP) problem. Clearly, arrival and service rates in such systems suffer from inherent uncertainty. A recent study addressed this issue by formulating a Robust Counterpart for SCLP models with budgeted uncertainty which provides a solution in terms of processing rates. This solution is transformed into a sequencing policy. However, in cases where servers can process several jobs simultaneously, a sequencing policy cannot be implemented.  In this paper, we propose to use in these cases a a resource allocation policy, namely,  the proportion of server effort per class. We formulate Robust Counterparts of both processing rates and server-effort uncertain models for four types of uncertainty sets: box, budgeted, one-sided budgeted, and polyhedral. We prove that server-effort model provides a better robust solution than any algebraic transformation of the robust solution of the processing rates model. Finally, to get a grasp of how much our new model improves over the processing rates robust model, we provide results of some numerical experiments.
\end{abstract}

\tikzstyle{server}=[%
rectangle,
minimum height=5cm,
text height=0.75cm,
text depth=.5cm,
text width=1cm,
inner xsep=1em,
inner ysep=1.5em,
text centered,
draw=black!50
]
\tikzstyle{server-horiz}=[%
rectangle,
text height=0.75cm,
text depth=.5cm,
text width=1cm,
inner xsep=1em,
inner ysep=1.5em,
text centered,
draw=black!50
]
\tikzstyle{server-label}=[%
fill=white,
font=\scriptsize
]
\tikzstyle{vserver}=[%
rectangle,
rounded corners,
minimum height=3cm,
text height=0.75cm,
text depth=.5cm,
text width=1cm,
inner xsep=1em,
inner ysep=1em,
text centered,
draw=black!50,
fill=black,
fill opacity=0.1
]
\tikzstyle{buffer}=[%
rectangle,
minimum height=1cm,
text height=0.75cm,
text depth=.5cm,
text width=1cm,
text centered,
font=\scriptsize,
inner sep=0pt,
draw=red!20,
fill=red!20
]
\tikzstyle{buffer-label}=[%
font=\tiny,
yshift=-0.2cm
]
\tikzstyle{input}=[%
rectangle,
minimum height=1cm,
text height=0.75cm,
text depth=.5cm,
text width=1.5cm,
text centered,
inner sep=0pt,
draw=white!0,
fill=white!0
]
\tikzstyle{output}=[%
rectangle,
minimum height=1cm,
text height=0.75cm,
text depth=.5cm,
text width=1cm,
text centered,
inner sep=0pt,
draw=white!0,
fill=white!0
]
\tikzstyle{task}=[%
rectangle,
minimum height=1cm,
text height=0.75cm,
text depth=.5cm,
text width=1cm,
text centered,
inner sep=0pt,
draw=black!50,
fill=orange!20
]
\tikzstyle{taskflow}=[%
semithick,
blue,
below,
->,
>=stealth
]

\section{Introduction}
\label{sec:introduction}
In multi-class processing networks \textit{items}
of different \textit{types} arrive at the system, where they are processed
in one or several ways,
then follow
individual paths through various \textit{service stations}. Each service station can perform \textit{jobs} of several different \textit{classes}, such that each job processes items of a specific type. Once processed, items either change type or leave the system.  To optimize system performance, one must control admissions, routing and sequencing of the items throughout the system \cite{Harrison1988}, \cite{Wein1992}, \cite{Kelly1993}, \cite{Dai1995}, \cite{Bramson2008}, \cite{Meyn2008}.
Such operational models are used in a number of applicative domains including manufacturing systems, multiprocessor
computer systems, communication networks, data centers, and
sensor networks. 
The straightforward formulation of such networks produces large stochastic dynamic programming problems that are computationally intractable. To overcome this, one can  approximate the network with a fluid model.
In order to find optimal solutions, we can formulate these networks as a specially structured class of continuous linear programs called Separated Continuous Linear Programs (SCLPs),
which 
has been studied over the last few decades.
Several algorithms  \cite{Pullan1993}, \cite{Luo1998}, \cite{Fleischer2005}, \cite{Weiss2008}, \cite{Bampou2012}, \cite{Shindin2018}, \cite{Shindin2021} have been developed to solve SCLPs and their generalizations. Some of these algorithms allow finding an exact solution. In particular, Weiss \cite{Weiss2008} developed the SCLP-simplex algorithm that provides an exact solution of SCLP problems in a finite bounded number of steps. 
Shindin and Weiss \cite{Shindin2014}, \cite{Shindin2015}, \cite{Shindin2018} extended this algorithm to more general M-CLP problems allows modeling impulse controls and solving general SCLP without additional assumption on the problem data.
Shindin et al \cite{Shindin2021} provide an efficient implementation of the revised SCLP-simplex algorithm which finds exact solutions for SCLP problems with hundreds of servers and thousands of job classes in reasonable time.
Nazarathy and Weiss \cite{Nazarathy2009} keep the deviations from the SCLP-generated model stable using a maximum pressure policy,
which provide \textit{stable} long-term solutions to fluid approximations to multi-class processing networks \cite{Dai2005}

Clearly, in a real-life application the arrival and processing rates of the various items may be uncertain and, moreover can change over  time.  There are two possible treatments to address this problem:
\begin{compactitem}
\item[-] Stochastic fluid model considered by Cassandras et al \cite{Cassandras2002} allows building gradient estimators of model parameters and then uses these estimators to build parameterized optimal control policies. This approach is mainly used for perturbation analysis and is computationally intractable for a large-scale processing networks. 
\item[-] Robust optimization approach deals with uncertain model parameters residing in a bounded {\em uncertainty set} and optimizes against the worst-case realization of the parameters within this set. Robust optimization treats the uncertainty in a deterministic manner and produces a tractable representation of the uncertain problem known as {\em robust counterpart}. In the context of fluid models this approach was considered by Bertsimas et al. \cite{bertsimas2014robust} who developed a robust counterpart for SCLP with a one-sided budgeted uncertainty set. The robust counterpart in \cite{bertsimas2014robust} is also an SCLP problem and thus can be solved by the algorithms discussed earlier.
\end{compactitem}

An important aspect of choosing a model formulation for a real-life problem is how to interpret its solution to a policy for the problem at hand. For example, for several types of processing networks, the model's solution should be transformed into a sequencing policy, namely, assigning priority to job classes so that when a server is done processing a job it will serve the next job from the class having highest priority. 
In the model suggested in \cite{bertsimas2014robust}, the system's control is formulated as processing rates for each job class. The resulting SCLP solution is used to assign priorities for the different job classes according to their processing rates at time $t=0$. However, in cases where servers can process several jobs simultaneously, we do not face a sequencing problem, but a resource allocation problem, namely, how much of the server resources should be dedicated to each class. The idea of controlling the proportion of server effort in fluid models has been used by \cite{Bramson_1996,Bramson_1998}. In this paper we consider a {\em server-effort fluid model}, where the control is formulated in terms of proportion of the server resources dedicated to serving specific classes of jobs. 

The contributions of the paper are as follows. 
(1) In \Cref{subsec:rc} we derive robust counterparts for both processing-rates and server-effort controlled fluid models under the different uncertainty sets discussed in \Cref{sec:fluid-processing-networks}.
(2) We compare the models under uncertainty and show that the server-effort model produces the same or better solution than the processing-rates model in \Cref{subsec:compare}.
(3) We demonstrate the performance improvement using a numerical comparison of these models for a processing network which can be easily transformed between these models in \Cref{sec:results}.

\section{Fluid Processing Networks}\label{sec:fluid-processing-networks}
In what follows, we make several assumptions. First, we consider only networks where each job class $j$ can be served by a single server denoted by $s(j)$, although each server can perform several classes of jobs. Second, all items of the specific type are stored in the dedicated buffer: when a job of a certain class is performed it takes an item from the buffer and the processed item either leaves the system or moves to another buffer for further processing. Therefore the arrival rate of a certain items is influenced by both the external arrival rate and the processing rates of the servers.  Finally, without loss of generality we also assume that buffer size is infinite. For the fluid approximation we consider fluids instead of individual items and flows instead of job classes.

\subsection{Example: A Criss-Cross Network}\label{ex:criss-cross-network}
    Let us consider the Criss-Cross network presented in \cite{bertsimas2014robust} and depicted in \Cref{fig:criss-cross-network}. In this network we have three types of items which are processed by jobs of the corresponding classes, so that we have one-to-one correspondence between item types and job classes. This network comprises two servers: The first server, $S_1$,
    contains two buffers $B_1$ and $B_2$ for items of types 1 and 2 respectively, that arrive from outside with rates $\lambda_1, \lambda_2$.
    Once a class 1 job is complete, it produces an item of type 3 that enters the buffer $B_3$ on server $S_2$. 
    For each job class $j$,
    let $\mu_j$ be the maximal service rate
    in jobs per unit time if the
    corresponding server $s(j)$ processes this class at full capacity.
    At time $t$, $x_k(t)$ is the quantity of items in the buffer $k$.
    
    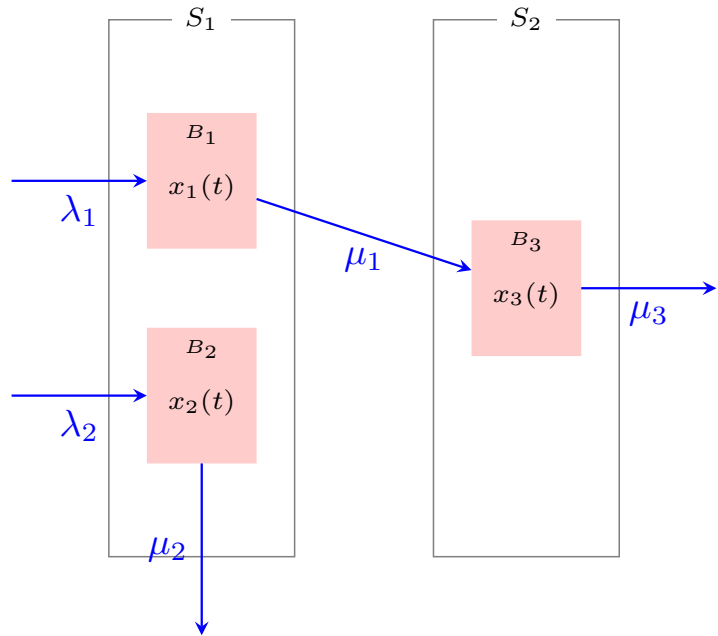
\begin{figure}[h!t]
        \centering
        \resizebox{0.8\columnwidth}{!}{
            \begin{tikzpicture}[node distance=1cm]
                \node[buffer] (B1) at (0,0) {$x_1(t)$};
                \node[buffer-label] at (B1.north) {$B_1$};
                \node[buffer] at (0, -2) (B2) {$x_2(t)$};
                \node[buffer-label] at (B2.north) {$B_2$};
                \node[buffer] at (3 , -1) (B3) {$x_3(t)$};
                \node[buffer-label] at (B3.north) {$B_3$};
                \node[input] (I1) [left=1.25cm of B1] {};
                \node[input] (I2) [left=1.25cm of B2] {};
                \node[output] (O2) [below=1.6cm of B2] {};
                \node[output] (O3) [right=1.25cm of B3] {};
                \node[server,fit=(B1)(B2)] (S1) {};
                \node[server-label] at (S1.north) {$S_1$};
                \node[server,fit=(B3)] (S2) {};
                \node[server-label] at (S2.north) {$S_2$};
                \path[taskflow] (I1) edge node {$\lambda_1$} (B1);
                \path[taskflow] (I2) edge node {$\lambda_2$} (B2);
                \path[taskflow] (B1) edge node {$\mu_1$} (B3);
                \path[taskflow, anchor=east] (B2) edge node {$\mu_2$} (O2);
                \path[taskflow] (B3) edge node {$\mu_3$} (O3);
            \end{tikzpicture}
        }
        \vspace{-3em}
        \caption[Criss-cross network]{
            \label{fig:criss-cross-network}
            A Criss-cross network.
             }
        \vspace{-1em}
    \end{figure}
    
    We formulate this queuing network as a fluid model by following:
    Let $\alpha_k = x_k(0)$ and $\lambda_k$ be the initial amount of fluid and the exogenous arrival rate of the fluid to the buffer $k$ respectively. Flow of class $j=k$ empties buffer $k$ with maximum rate $\mu_j$. Let $x_k(t)$ denote the total amount of fluid in the buffer $k$ at time $t$. Here $x_k(t)$ are the system's state variables. The actual decision variables are the system's control variables relating to processing of different flows by the corresponding server. 
    
    There are two ways we can model this fluid network:
    \begin{compactenum}[(A)]
        \item We can control $u_j(t)$ - the actual processing rate of flow unit of class $j$ per unit time at time $t$, so that $0\leq u_j(t) \leq \mu_j$.
        In this case the dynamics of the system can be expressed as:
        \begin{equation*}
            \label{eqn:dyn1}
            \begin{array}{ll}
                x_1(t) = \alpha_1 +\lambda_1 t - \int_0^t u_1(s)ds & \forall t,\\
                x_2(t) = \alpha_2 +\lambda_2 t - \int_0^t u_2(s)ds & \forall t,\\
                x_3(t) = \alpha_3 + \int_0^t u_1(s)ds - \int_0^t u_3(s)ds & \forall t, 
            \end{array}   
        \end{equation*}                  
        where the sum of the actual processing rates should not exceed the total server capacity:
        \begin{equation*}
            \label{eqn:eff1}
            \begin{array}{ll}
                \frac{u_1(t)}{\mu_1} + \frac{u_2(t)}{\mu_2} \le 1 & \forall t, \\
                \frac{u_3(t)}{\mu_3} \le 1 & \forall t.
            \end{array}
        \end{equation*}
        \item  Alternatively, we can control $0\leq \eta_j(t)\leq 1$ which is the dimensionless proportion of the server effort dedicated to flow $j$ at time $t$. Clearly, the actual processing rates are proportional to the dedicated server efforts: $u_j(t) = \eta_j(t)\mu_j$. 
        In this case the dynamics of the system can be expressed as:
        \begin{equation*}
            \label{eqn:dyn2}
            \begin{array}{ll}
            x_1(t) = \alpha_1 +\lambda_1 t - \int_0^t \mu_1 \eta_1(s)ds & \forall t,\\
            x_2(t) = \alpha_2 +\lambda_2 t - \int_0^t \mu_2 \eta_2(s)ds & \forall t,\\
            x_3(t) = \alpha_3 + \int_0^t \mu_1 \eta_1(s)ds - \int_0^t \mu_3 \eta_3(s)ds & \forall t, 
            \end{array}
        \end{equation*}
        where the total server effort should not exceed $1$:
        \begin{equation*}
            \label{eqn:eff2}
            \begin{array}{ll}
                \eta_1 (t) + \eta_2 (t) \le 1 & \forall t, \\
                \eta_3 (t) \le 1 & \forall t.
            \end{array}
        \end{equation*}
    \end{compactenum}
\subsection{General processing networks}
Consider a processing network with $I$ servers and $J$ different job classes.
We also have $k=1,\dots, K$ buffers
so that jobs of class $j$ process items of type $k = k(j)$.
Each job class $j$ has an associated server $i = s(j)$ that processes jobs of this class.
Items processed by the job class $j$ routed to the buffer $k$ with probability $p_{j,k}$ or leave the system with probability $1 - \sum_k p_{j,k}$.

In the associated fluid model the flow of class $j$ empties buffer $k=k(j)$ and is processed by the server $i=s(j)$. After the processing fluid comes to the buffer $l$ in proportion $p_{j,l}$ or leaves the system in proportion  $1 - \sum_l p_{j,l}$. To simplify the notation we define matrix $G$ of size $K \times J$  and set $G_{k,j} = 1$ for $k=k(j)$ and $G_{k,j} = - p_{j,k}$ for $k \ne k(j)$. 
We let $\lambda_k$ be the rate of external arrivals to the buffer $k$ and $\alpha_k$ be initial amount of fluid in the buffer $k$.

Let $\mu_j$ be the processing rate of flow $j$ when the server devotes all of its effort to this flow, and $\tau_j = 1/\mu_j$ be mean service time per unit of flow $j$. 
Similarly to example in \ref{ex:criss-cross-network}, for cost functions $f,g$ on $\RR{K},\RR{J}$ we can define two types of control policies:
\begin{compactenum}[(A)]
    \item \textbf{\Gls{modelA}} our decisions are actual processing rates $u_j(t) = \eta_j(t)\mu_j$. This is the approach taken in \cite{bertsimas2014robust}. In this case the optimal control can be found by solving following optimization problem:
\begin{align}
 \min_{x(t), u(t)} & \int_0^T f(x(t)) + g(u(t)) \dd t , \nonumber \\
    \label{eqn:dyn1g}
 s.t.\quad & \int_0^t  \sum_{j} G_{k,j} u_{j}(s)\dd s + x_k(t) = \alpha_k + \lambda_k t\;
            \forall k,t, \\
    \label{eqn:eff1g}     
      & \sum_{j: s(j)=i} \tau_j u_j(t) \le 1\; \forall i,t, \\
      & u(t), x(t) \ge 0. \nonumber
\end{align}
    \item \textbf{\Gls{modelB}} our decisions are proportions of the server effort $\eta_j(t)$. In this case the optimal control can be found by solving following optimization problem, where $\circ$ denotes elementwise multiplication:
    \begin{align}
 \min_{x(t), \eta(t)}  & \int_0^T f(x(t)) + g(\mu \circ \eta(t)) \dd t , \nonumber \\
    \label{eqn:dyn2g}
 s.t. \quad&  \int_0^t  \sum_{j} G_{k,j} \mu_j \eta_{j}(s)\dd s + x_k(t) = \alpha_k + \lambda_k t\;
            \forall k,t, \\
    \label{eqn:eff2g}             
      &  \sum_{j: s(j)=i} \eta_j(t) \le 1\; \forall i,t, \\
      & \eta(t), x(t) \ge 0. \nonumber
\end{align} 
 \end{compactenum}
One can see that these models are interchangeable,
meaning that optimal solutions can be obtained each from other by substituting
$\eta_j(t) = {\tau_j}{u_j(t)} ~\forall j$, or $u_j(t) = \mu_j \eta_j(t) ~\forall j$.
In the next section we show that the equivalence between solutions the models holds only for the case when service rates are certain.

\section{Fluid Processing Networks under Uncertainty}
\label{sec:uncertainty}
\subsection{Modelling Uncertainty}
\label{subsec:Uncertainty-model}
In practice arrival rates $\lambda_k$, service rates $\mu_j$, and/or mean service times $\tau_j = 1/\mu_j$ can be uncertain, meaning that we do not know their exact values at optimization time. Moreover these parameters can also change over time. 
We assume that the unknown parameters belong to some bounded uncertainty set $\uncertainset{U}$. The uncertainty set is assumed to be known (or estimated from historical data) a priori.

For a given uncertainty set $\uncertainset{U}$,
one can define a new optimization problem whose solution is feasible under all possible realizations
of the uncertain parameters 
Such a problem is known as the {\em robust counterpart}.
The formulation of the robust counterpart depends both on the initial problem and the the uncertainty set. 
The robust approach is conservative in that it protects against the worst-case scenario. 
It should be noted that a robust counterpart is created per constraint, so constraints are treated individually even if the same uncertain parameters are shared between them. This is conservative because the worst-case scenario for one constraint can be different than for another.

Let us denote all uncertain time-varying parameters of a problem by $\theta_\ell(t), \ell=1,\dots, L$. We formulate each of these parameters as $\theta_\ell(t) =\oo{\theta}_\ell + \ot{\theta}_\ell\zeta_\ell(t)$ where $\oo{\theta}_\ell$ is called the \textit{nominal} value of the uncertain parameter $\theta_\ell(t)$, $\ot{\theta}_\ell$  denotes the maximal \textit{deviation} of $\oo{\theta}_\ell$ from its nominal value, and $\zeta_\ell(t)$ is called the \textit{perturbation}. 
To keep robust counterpart simple and computationally tractable we model uncertainty using one of the following uncertainty sets:
\begin{compactitem}
    \item[-] {\em Box uncertainty set.} The perturbations affecting the uncertain parameters reside in a box : $|\zeta_\ell(t)|\le 1 \forall t, \ell=1.\dots,L$. This means that  
       $\theta_\ell(t)$ can take values in the interval $[\oo{\theta}_\ell - \ot{\theta}_\ell, \oo{\theta}_\ell + \ot{\theta}_\ell]$. Such a set indicates that the perturbations are independent of each other.
     \item[-] {\em Budgeted uncertainty set.} Consider, that in addition to the box uncertainty for some group of parameters $\L$ we have an additional restriction on the total magnitude of the perturbations: $\sum_{\ell \in \L} |\zeta_\ell(t)| \le \Gamma$. When $\Gamma \le \|\L\|$ is called the uncertainty budget. As robust counterpart is formulated per constraint, the uncertainty budget also refers to the specific constraint and refers to the uncertain parameters related to this constraint.
     \item[-] {\em Budgeted one-sided uncertainty set.}
     Similar to the budgeted uncertainty set, but here one consider only positive perturbations $0 \le \zeta_\ell(t) \le 1 \forall t$ and restrict on the total sum by: $\sum_{\ell \in \L} \zeta_\ell(t) \le \Gamma$.
     \item[-] {\em Polyhedral uncertainty set.}
    The perturbations reside in a polyhedron, meaning  $\{\zeta(t) | D \zeta(t) +d \ge 0\} \forall t$, where $D\in R^{M \times L}$ and $d\in R^M$ for some $M > 0$.
\end{compactitem}
Note that level of conservativeness of the robust counterpart depends on the volume of the uncertainty set. Thus, in general, it is desirable to use a smaller uncertainty set.


The uncertain parameters referring to arrival rates $\lambda_k$ appear in both \gls{modelA} and \gls{modelB}. The other uncertain parameters vary between the models: \gls{modelA} is affected by uncertainty in mean service time per flow unit $\tau_j(t) = \oo{\tau}_j + \ot{\tau}_j\zeta_j(t)$, 
while in \gls{modelB} the uncertain parameters are service rate flow per unit time $\mu_j(t) = \oo{\mu}_j - \ot{\mu}_j\zeta_j(t)$.
Since $\tau_j=1/\mu_j$, the shape of the uncertainty set affecting $\tau_j$ may vary from the shape of the uncertainty set affecting $\mu_j$. Moreover, those parameters in different models are related to the different constraints and hence, even though in their deterministic form \gls{modelA} and \gls{modelB} are equivalent, their robust counterparts, and therefore the solutions they yield, may differ.

\subsection{Robust counterparts}
\label{subsec:rc}
As mentioned in \Cref{subsec:Uncertainty-model} formulation of a robust counterpart of an uncertain constraint depends both on the original constraint and on the shape of the uncertainty set its uncertain parameters reside in.
In this section we present tractable robust counterparts for the different constraints of the \gls{modelA} and the \gls{modelB} for the
uncertainty sets described in \Cref{subsec:Uncertainty-model}.
Proofs of the propositions below can be found in the Appendix.

As the robust counterpart is formulated per constraint, we consider the uncertainty set affecting the uncertain parameters for each constraint individually. Therefore, for system dynamics constraints in the \gls{modelA} (\Cref{eqn:dyn1g}) 
which contain only a single uncertain parameter - arrival rate $\lambda_k$, all uncertainty sets
are equivalent to a box uncertainty, resulting in the same robust counterpart:

\begin{prop}
    \label{prop:balance-model-a}
    The robust counterpart of \Cref{eqn:dyn1g}
    for all uncertainty sets shapes
    is
    \begin{equation}
    \label{eqn:robust-counterpart-equation-balance-a}
        \alpha_k + (\oo{\lambda}_k - \ot{\lambda}_k) t - 
            \!\! \sum\limits_{j} \int_0^t G_{kj} u_{j} (s)\dd s \ge 0, \forall k,t \text{.}
    \end{equation}
\end{prop}
 
For server capacity constraints in \gls{modelA} the uncertain parameters are mean service time per flow unit $\tau_j$ related to all flows processing by server $s(j)$.

\begin{prop}
    \label{prop:server-model-a}
    The robust counterparts of \Cref{eqn:eff1g}
    are given in \Cref{tab:robust-counterpart-equation-server-capacity-a}.
\end{prop}

\begin{table}[!hptb]
    \centering
    \begin{tabular}{ll}
        \toprule
        Uncertainty set & Robust Counterpart \\
        \midrule
        Box & $\sum\limits_{j: s(j)=i} \left( \oo{\tau}_j + \ot{\tau}_j \right) u_j(t) \le 1$ \\
        \midrule[0.02\lightrulewidth] \\
        \multirow[c]{3}{1.5cm}{One-sided \& Budgeted}
            & $\Gamma_{i} \beta_{i}(t)+\sum\limits_{j: s(j)=i}\left(\oo{\tau}_{j}
                u_{j}(t)+\gamma_{j}(t)\right) \leq 1,$ \\
            & $\beta_{i}(t)+\gamma_{j}(t)-u_{i}(t) \ot{\tau}_{i} \geq 0,\; j=1,\dots, J 
                    $ \\
                    & $\beta(t), \gamma(t) \ge 0$ \\
        \midrule[0.02\lightrulewidth] \\
        \multirow[c]{2}{1.5cm}{Polyhedral}
            & $\sum\limits_{j} \oo{\tau}_{j} u_{j}(t)
                - \sum\limits_{m} \delta_{m}(t) d_m \leq 1,$ \\
            & $ \sum\limits_{m} D_{mj} \delta_{m} = u_{j}(t) \ot{\tau}_{j}, \; j=1,\dots, J,\;  \delta \ge 0$. \\
        \bottomrule
    \end{tabular}
    \caption{\Gls{modelA}: Robust counterpart to server capacity constraints \Cref{eqn:eff1g}}
    \label{tab:robust-counterpart-equation-server-capacity-a}
\end{table}

For the \gls{modelB} the uncertain parameters in each of the constraints (\ref{eqn:dyn2g}) are service rates per unit time of all related flows $\mu_j$ and also exogenous arrival rate $\lambda_k$, while there is no uncertainty affecting \Cref{eqn:eff2g}. 
\begin{prop}
    \label{prop:balance-model-b}
    The robust counterparts of \Cref{eqn:dyn2g}
    are given in \Cref{tab:robust-counterpart-equation-balance-b}.
\end{prop}

\begin{table*}
    \centering
    \begin{tabular}{ll}
        \toprule
        Uncertainty set & Robust Counterpart \\
        \midrule
        Box & $
            \alpha_k + (\oo{\lambda}_k - \ot{\lambda}_k) t -
            \sum\limits_{j: k(j)=k} \!\! \int_0^t (\oo{\mu}_j + \ot{\mu}_j) 
            \eta_{j} (s)\dd s
            + \sum\limits_{j:k(j) \ne k} \!\! \int_0^t G_{kj} (\oo{\mu}_{j} - \ot{\mu}_{j}) \eta_{j}(s) \dd s \ge 0$ \\
        \midrule[0.02\lightrulewidth] \\
        \multirow[c]{2}{1.25cm}{One-sided}
            & $\alpha_k + (\oo{\lambda}_k - \ot{\lambda}_k) t
                - \sum\limits_{j} \!\! G_{kj} \oo{\mu}_j
                    \int_0^t  \eta_{j} (s)\dd s
                + \sum\limits_{i} \!\! \Gamma_i \beta_{ki}(t) 
                + \sum\limits_i \!\sum\limits_{j:s(j)=i} \gamma_{kij}(t) \geq 0$ \\
            & $\gamma_{kij}(t) + \beta_{ki}(t) 
                \geq - G_{kj}\ot{\mu}_j \int_0^t \eta_j(s) \dd s
                ~\forall (i, j:s(j) = i), \; \gamma(t), \beta(t) \ge 0$ \\
        \midrule[0.02\lightrulewidth] \\
        \multirow[c]{2}{1.25cm}{Budgeted} 
            & $\alpha_k + (\oo{\lambda}_k - \ot{\lambda}_k) t
                - \sum\limits_{j} \!\! G_{kj} \oo{\mu}_j
                    \int_0^t  \eta_{j} (s)\dd s 
                + \sum\limits_{i} \!\! \Gamma_i \beta_{ki}(t) 
                + \sum\limits_i \!\sum\limits_{j:s(j)=i} \gamma_{kij}(t) \geq 0$ \\
            & $\beta_{ki}(t) + \gamma_{kij}(t) - 2 \delta_{kij}
                \geq - G_{kj}\ot{\mu}_j \int_0^t \eta_j(s) \dd s$,  \\
            & $\delta_{kij}(t)  +  G_{kj}\ot{\mu}_j \int_0^t \eta_j(s) \dd s \ge 0
               ~\forall (i, j:s(j) = i), \; \gamma(t), \delta(t), \beta(t) \ge 0  $ \\
        \midrule[0.02\lightrulewidth] \\
        \multirow[c]{2}{1.25cm}{Polyhedral}
            & $\alpha_k + (\oo{\lambda}_k - \ot{\lambda}_k) t
                - \sum\limits_{j} \!\! G_{kj} \oo{\mu}_j
                    \int_0^t  \eta_{j} (s)\dd s 
                + \sum\limits_{m} \!\! d_{m} \delta_{km}(t) \geq 0$ \\
            & $- \sum\limits_{m}  D_{mj} \delta_{km}
                = G_{kj} \ot{\mu}_j \int\limits_0^t \eta_{j}(s) \dd s \; \forall j, \; \delta(t) \geq 0$\\
        \bottomrule
    \end{tabular}
    \caption{\Gls{modelB}: Robust counterpart to system dynamics constraints \Cref{eqn:dyn2g}}
    \label{tab:robust-counterpart-equation-balance-b}
\end{table*}


The
objective functionals also depend on uncertain parameters, and thus their robust counterparts should be built w.r.t a specific $f(\cdot), g(\cdot)$ and specific uncertainty set. In the special case, where the objective is to minimize holding costs in the buffers we have $f(x(t)) = c^\transpose x(t)$ and $g(u(t)) = 0$. In this case the robust counterparts of the objective has shape of SCLP objective and hence resulting problem is tractable. Substituting the system dynamic equations (\ref{eqn:dyn1g}) or (\ref{eqn:dyn2g}) into the objective,
one can see that value of the objective functional depends on uncertainty in arrival rates $\lambda$. However, as cost related to the arrival rates independent of our control it is possible to calculate the maximal arrival cost $\Lambda$ independently from the optimization of the fluid model as specified in \Cref{tab:robust_lambda}.

\begin{prop}
    \label{prop:robust-objective-modA}
    Let  $f(x(t))=c\Tt x(t), c \ge 0$ , $g(u(t))=0$ and
        $\Lambda$ taken from \Cref{tab:robust_lambda} then the robust counterpart of objective functional for \gls{modelA} is equivalent to: 
    \begin{align}\label{eqn:robust-counterpart-objective-a}
        \min_{u(t),z(t)} \quad & \int_{0}^{T} z(t) dt \nonumber \\
        s.t. \quad & \Lambda t + c\Tt \alpha  - \int_0^t c\Tt G u (s)\dd s \le z(t), \forall t \text{,} 
    \end{align}
\end{prop}

\begin{table}[phtb]
    \centering
    \begin{tabular}{ll}
        \toprule
        Uncertainty set & Maximal arrival cost $\Lambda$ \\
        \midrule
        Box & $\Lambda = c\Tt(\oo{\lambda} + \ot{\lambda})$ \\
        \midrule[0.02\lightrulewidth] \\
        \multirow[c]{2}{1.5cm}{One-sided \& Budgeted}
            & $\Lambda = c\Tt \oo{\lambda} + \min_{\beta^*, \gamma^*} \Gamma^* \beta^* + \sum_k \gamma^*_k$  \\
            & $\beta^*+\gamma_{k}^*- c_k \ot{\lambda}_k\geq 0, \; k= 1,\dots, K,\; \beta^*,\gamma^* \ge 0
                    $\\
        \midrule[0.02\lightrulewidth] \\
        \multirow[c]{2}{1.5cm}{Polyhedral}
            & $ \Lambda = c\Tt \oo{\lambda} - \max_{\delta^*} \sum\limits_{m} \delta^*_{m} d_m,$ \\
            & $ D\Tt \delta^* = - c_k \ot{\lambda}_k , \; k= 1,\dots, K, \;\delta^* \geq 0$ \\
        \bottomrule
    \end{tabular}
    \caption{Robust arrival rates}
    \label{tab:robust_lambda}
\end{table}

  \begin{table*}[htbp] 
     \centering
     \begin{tabular}{p{1.5cm}l}
         \toprule
         Uncertainty & Robust Counterpart \\
         \midrule
         Box & $z(t)\ge \Lambda t + \sum_k c_k
             \left[ \alpha_k - 
                 \!\!  \sum\limits_{j: k(j)=k} \!\! \int_0^t (\oo{\mu}_j - \ot{\mu}_j) 
            \eta_{j} (s)\dd s
            + \sum\limits_{j:k(j) \ne k} \!\! \int_0^t G_{kj} (\oo{\mu}_{j} + \ot{\mu}_{j}) \eta_{j}(s) \dd s \right] $ \\
         One-sided & $z(t)\ge \Lambda t + \sum_k c_k
             \left[ \alpha_k  - 
                 \!\! \int_0^t \sum_j G_{kj} \oo{\mu}_j \eta_j (s)\dd s +
                     \sum_i \Gamma_i \beta_{ki}(t) + \sum\limits_i \!\sum\limits_{j:s(j)=i} \gamma_{kij}(t) \right] $ \\
                & $\gamma_{kij}(t) + \beta_{ki}(t) 
                \geq G_{kj}\ot{\mu}_j \int_0^t \eta_j(s) \dd s
                ~\forall (i, j:s(j) = i), \; \gamma_{kij}(t), \beta_{ki}(t) \ge 0$ \\  
        Budgeted & $z(t)\ge \Lambda t + \sum_k c_k
             \left[ \alpha_k  -  
                 \!\! \int_0^t \sum_j G_{kj} \oo{\mu}_j \eta_j (s)\dd s +
                     \sum_i \Gamma_i \beta_{ki}(t) + \sum\limits_i \!\sum\limits_{j:s(j)=i} \gamma_{kij}(t) \right]$ \\
                & $\beta_{ki}(t) + \gamma_{kij}(t) - 2 \delta_{kij}
                \geq G_{kj}\ot{\mu}_j \int_0^t \eta_j(s) \dd s$, \\
                & $\delta_{kij}(t)  -  G_{kj}\ot{\mu}_j \int_0^t \eta_j(s) \dd s \ge 0
               ~\forall (i, j:s(j) = i), \; \gamma(t), \delta(t), \beta(t) \ge 0  $ \\
         Polyhedral & $z(t)\ge \Lambda t + \sum_k c_k
             \left[ \alpha_k  -  
                 \!\! \int_0^t \sum_j G_{kj} \oo{\mu}_j \eta_j (s)\dd s  - \sum\limits_{m} \!\! d_{m} \delta_{km}(t) \right] $ \\
               & $\sum\limits_{m} D_{mj} \delta_{km}
                = G_{kj} \ot{\mu}_j \int\limits_0^t \eta_{j}(s) \dd s  \; \forall j, \; \delta(t) \geq 0$ \\   
         \bottomrule
     \end{tabular}
     \caption{Robust objective functions for \gls{modelB}}
     \label{tab:robust-objective-functions}
 \end{table*}

\begin{prop}
 \label{prop:robust-objective-modB}
    Let  $f(x(t))=c\Tt x(t), c \ge 0$ , $g(u(t))=0$ then the robust counterpart of objective functional for \gls{modelB} is equivalent to: 
    \begin{align*}
        \min_{\eta(t),z(t)} \quad & \int_{0}^{T} z(t) dt \\
        s.t.  &\text{ Constraints: (\Cref{tab:robust-objective-functions})}
    \end{align*}
\end{prop}
The proofs of \Cref{prop:robust-objective-modA} and \Cref{prop:robust-objective-modB} are very similar to those of \Cref{prop:balance-model-a} and \Cref{prop:balance-model-b}, respectively.

Based on \Cref{prop:balance-model-a,prop:server-model-a,prop:robust-objective-modA}  the robust counterpart of \gls{modelA} is: 
\begin{equation}
\label{eqn:robust-modA}
\begin{array}{ll}    
        \displaystyle \min_{u(t),z(t)} & \int_0^T z(t) dt \\
      s.t. & \text{Constraints:} \text{ \Cref{eqn:robust-counterpart-objective-a,eqn:robust-counterpart-equation-balance-a}, \Cref{tab:robust-counterpart-equation-server-capacity-a}}, \\
        & u(t)\ge 0, 
 \end{array}       
\end{equation}
Based on \Cref{prop:balance-model-b,prop:robust-objective-modB} the robust counterpart of \gls{modelB} is: 
\begin{equation}
\label{eqn:robust-modB}
\begin{array}{ll}
   \displaystyle  \min_{\eta(t),z(t)} & \int_0^T z(t) dt  \\
     s.t. & \text{Constraints: } \text{\Cref{tab:robust-objective-functions,tab:robust-counterpart-equation-balance-b}, Eq. \ref{eqn:eff2g}}, \\
    & \eta(t) \ge 0 
 \end{array}       
\end{equation}

\subsection{Processing-rates model vs Server-effort model}
\label{subsec:compare}
The \gls{modelA} represents a case where the suitable control of the system at hand is its class specific processing rates $u_j(t)$ while the \gls{modelB} is used to describe cases where system is controlled by the proportion of resources dedicated to specific job classes $\eta_j(t)$. When no uncertainty affects the models, they are equivalent. In other words, any solution of the \gls{modelA} can be transformed into a solution of the \gls{modelB} using the relation $\eta(t) = \tau(t) u(t)$. In particular, if $u^*_j(t)$ is an optimal solution for the \gls{modelA}, then  $\eta^*(t) = \tau(t) u^*(t)$ is an optimal solution for the \gls{modelB}.
When uncertainty is introduced into the models, this is no longer possible in general. If $u_j(t)$ is a robust solution for the \gls{modelA} (e.g. of Problem (\ref{eqn:robust-modA})), then the value of $\eta(t) = \tau(t) u(t)$ depends on the uncertain parameters $\tau_j(t)$. However, we need a certain and feasible transformation to implement this solution as a control for the \gls{modelB} ((e.g. of Problem (\ref{eqn:robust-modB})). This applies also to the robust optimal solutions of the models.


Let us assume there exists a transformation $h:u(t)\to \eta(t)$ between robust solutions of \gls{modelA} and \gls{modelB}. That is to say, if $u(t)$ is a feasible solution of Problem (\ref{eqn:robust-modA})), then $h:u(t)\to \eta(t)$ is a feasible solution of Problem (\ref{eqn:robust-modB})).
In the following theorem we show that a robust optimal solution of \gls{modelB} is as good or better than any such transformation $h(u)$ of the robust optimal solution of \gls{modelA}.

\begin{thm}
\label{thm:model-better}
    Let $u^*(t)$ be an optimal solution of Problem (\ref{eqn:robust-modA}).
    Let $h$ be a transformation from the solution space of Problem (\ref{eqn:robust-modA})
    to the solutions space of Problem (\ref{eqn:robust-modB}),
    and let $\eta^*(t) = h(u^*(t))$.
    Then the optimal objective value of Problem (\ref{eqn:robust-modB}) is
    at least as good as that of $\eta^*(t)$.
\end{thm}

\begin{proof}
    $\eta^*(t) = h(u^*(t))$ where $u^*(t)$ is the optimal solution of Problem (\ref{eqn:robust-modA}),
    must be a feasible solution of Problem (\ref{eqn:robust-modA}). However, it is not necessarily  optimal.
\end{proof}
Note that \Cref{thm:model-better} applies to all uncertainty sets and objective functionals.



\section{Computational Results}
\label{sec:results}

\Cref{thm:model-better} indicates that the optimal robust solution of the \gls{modelB} is as good or better than the transformed optimal robust solution of \gls{modelA}.
In this section we perform a numerical study 
on two example systems to quantify how much optimal robust solution of \gls{modelB} outperforms the transformed optimal robust solution of \gls{modelA} in terms of holding cost. These systems have 10 and 20 servers respectively, with 10 job classes processed on each server and no internal inflows in order to have straightforward transformation rules.
The two models being compared handle uncertainty in arrival rates in the same manner,
Therefore we opted to simplify our example by looking only at uncertain service times $\tau(t)$. We consider a box uncertainty set $\tau(t) = \oo\tau +\epsilon\oo{\tau}\zeta(t)$ where  $\epsilon$ denotes the maximal deviation of actual service time $\tau(t)$ from its nominal value $\oo{\tau}$.
Now $\mu(t) = 1/\tau(t)$,
so $\frac{1}{(1+\epsilon)\oo{\tau}_j} \leq \mu_j(t) \leq \frac{1}{(1-\epsilon)\oo{\tau}_j}$. Solving for the midpoint of this box
gives us the equivalent uncertainty set $\oo{\mu} - \epsilon \oo{\mu}\zeta(t)$ centered around $\oo{\mu}_j = \frac{1}{\oo{\tau}_j(1-\epsilon^2)}$.

Parameter values for the example system are taken randomly from the values in \Cref{tab:results:parameters}.

\begin{table}[h!bt]
    \centering
    \begin{tabular}{ccc}
        \toprule
         Name& Symbol & Range \\
         \midrule
         nominal service rate & $\oo{\mu}_j = 1/\oo{\tau}_j$ & 5 - 25 \\
         nominal arrival rate &$\lambda_k$ & 2 - 5 \\
         initial buffer size&$\alpha_k$ & 10 - 20 \\
         holding cost&$c_k$ & 1 - 2 \\
         \bottomrule
    \end{tabular}
    \caption{Experimental parameters used in simulations.
        Values were randomly sampled from the given range.}
    \label{tab:results:parameters}
\end{table}

The general approach is to construct the robust counterparts for both a \gls{modelA} and a \gls{modelB} (meaning, Problems (\ref{eqn:robust-modA}), (\ref{eqn:robust-modB}) respectively) for the same
randomized set of parameters $\tau=1/\mu$, $\lambda$, $\alpha$, $c$,
Next, we solve these two SCLP problems to obtain their optimal robust controls $u^*(t)$ and $\eta^*(t)$.
We then create ten realizations of $\tau(t)=[\tau_1(t),\ldots,\tau_K(t)]$ using \Cref{eqn:results:tau_t} for each of the ten sets of random model parameters, for a total of 100 realizations of $\tau(t)$.
\Cref{fig:uncertain_arrival_times} depicts several examples of such realizations.
\begin{equation}
    \label{eqn:results:tau_t}
    \tau_k(t) = \tau_k + \frac{1}{4} \sum_{n=1}^4 {\sin(n \pi t + \phi_n), ~\phi_n \sim U[0,2\pi]}
\end{equation}
For each realization, we can compute the actual buffer sizes $x_k(t)$ in the system when applying the robust optimal control $\eta^*(t)$ and their sizes when applying the transformed control $\eta^{u}_j(t) = u^*_j(t)/(1-\tau_j)$.

These actual buffer sizes determine the actual holding cost for both these controls. Lastly, we average the difference between those objective values over all 100 realizations.
We repeat this procedure ten times, each time using a new set of random parameters, and compute the average difference between objective values over the ten parameters sets.

We ran this experiment for two different sized networks:
a) a smaller network with $I=10$ servers and $J=100$ job classes,
and b) a larger network with $I=20$ servers and $J=200$ job classes.
In both cases, the networks have no internal inflows,
meaning that $J=K$ and $G_{kj}=0$ when $j \neq k$.
For each of network,
we ran the experiment for five values of relative uncertainty: $\epsilon \in \{0.01,0.02,0.05,0.1,0.2\}$.
\Cref{alg:results:method} describes in detail the experimental flow.

\begin{algorithm}[h!t]
\caption{Method for numerical experiment}
\label{alg:results:method}
\begin{algorithmic}
\Require $I,J,K,\epsilon$
\For{$n_p = 1,\ldots,10$} \algorithmiccomment{random models}
    \State random $\tau=1/\mu$, $\lambda$, $\alpha$, $c$ \algorithmiccomment{from \Cref{tab:results:parameters}}
    \State $u^*(t) \gets$ \gls{modelA} using $(1+\epsilon)\tau$
    \State \qquad \algorithmiccomment{from \Cref{tab:robust-counterpart-equation-server-capacity-a}}
    \State $\eta^*(t) \gets$ \gls{modelB} using $[\tau(1-\epsilon)]^{-1}$
    \State \qquad \algorithmiccomment{worst case of $\mu$ for $\tau(t) \geq (1-\epsilon)\tau$}
    \For{$n_t = 1, \ldots, 10$} \algorithmiccomment{perturbations of $\tau$}
        \State $\tau(t) \gets \mbox{rand}(\tau)$ \algorithmiccomment{explained below}
        \State \textbf{\Gls{modelA}:}
        \State $\eta^{u}(t) \gets u^*(t) \tau (1-\epsilon)$ \algorithmiccomment{transformation of $u^*(t)$}
        \State $\hat{x}_{1} \gets \alpha + \lambda t - \int_0^t \eta^{u}(s)/\tau(s) \dd s$
        \State $z_1 \gets \int_0^T c^\transpose \hat{x}_{1}(t) \dd t$
        \State \textbf{\Gls{modelB}:}
        \State $\hat{x}_{2} \gets \alpha + \lambda t - \int_0^t \eta^*(s)/\tau(s) \dd s$
        \State $z_2 \gets \int_0^T c^\transpose \hat{x}_{2}(t) \dd t$
        \State \textbf{Relative difference:}
        \State $\Delta_{12} \gets (z_1 - z_2)/z_1$ \algorithmiccomment{relative difference}
    \EndFor
\EndFor
\end{algorithmic}
\end{algorithm}

Results,
which are summarized in \Cref{tab:results:mean_improvement} 
indicate that the relative improvement of robust control of \gls{modelB} over control derived from robust control of \gls{modelA} increases with the level of uncertainty.
The results also suggest that this improvement size is independent of network size.
{\em Note:} The computational costs of the the new model will be explored in a subsequent paper.
The python code for running the experiments can be found in \cite{Ship_SCLP-Python_2022},
a fork of \cite{Shindin_SCLP-Python_2021}.

\begin{figure}[bt]
    \centering
    \vspace{-1em}
    \begin{subfigure}[h]{0.48\columnwidth}
        \centering
        \includegraphics[width=\textwidth]{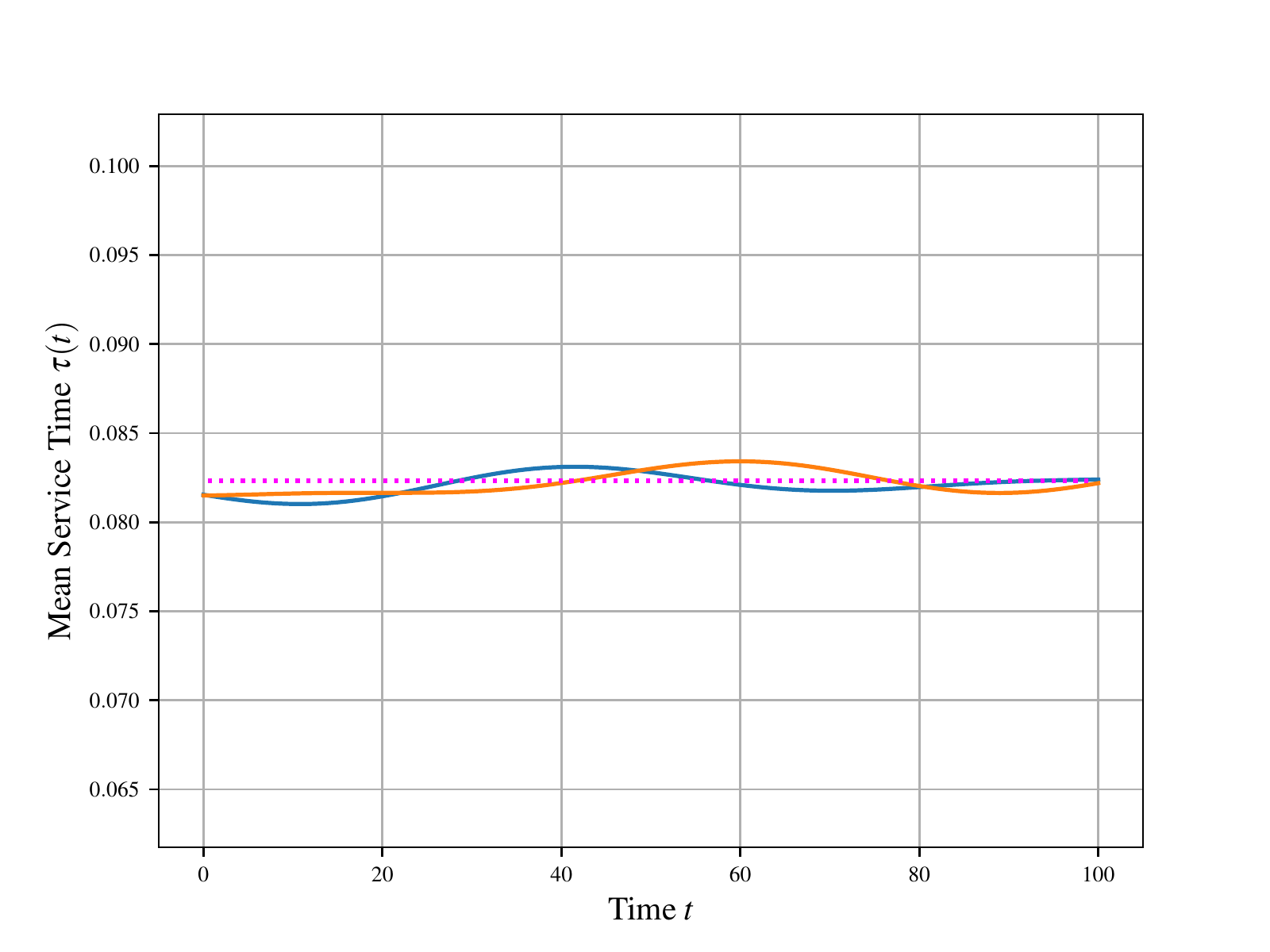}
        \caption{$\epsilon=0.02$}
        \label{fig:uncertain_arrivals:0.02}
    \end{subfigure}
    \hfill
    \begin{subfigure}[h]{0.48\columnwidth}
        \centering
        \includegraphics[width=\textwidth]{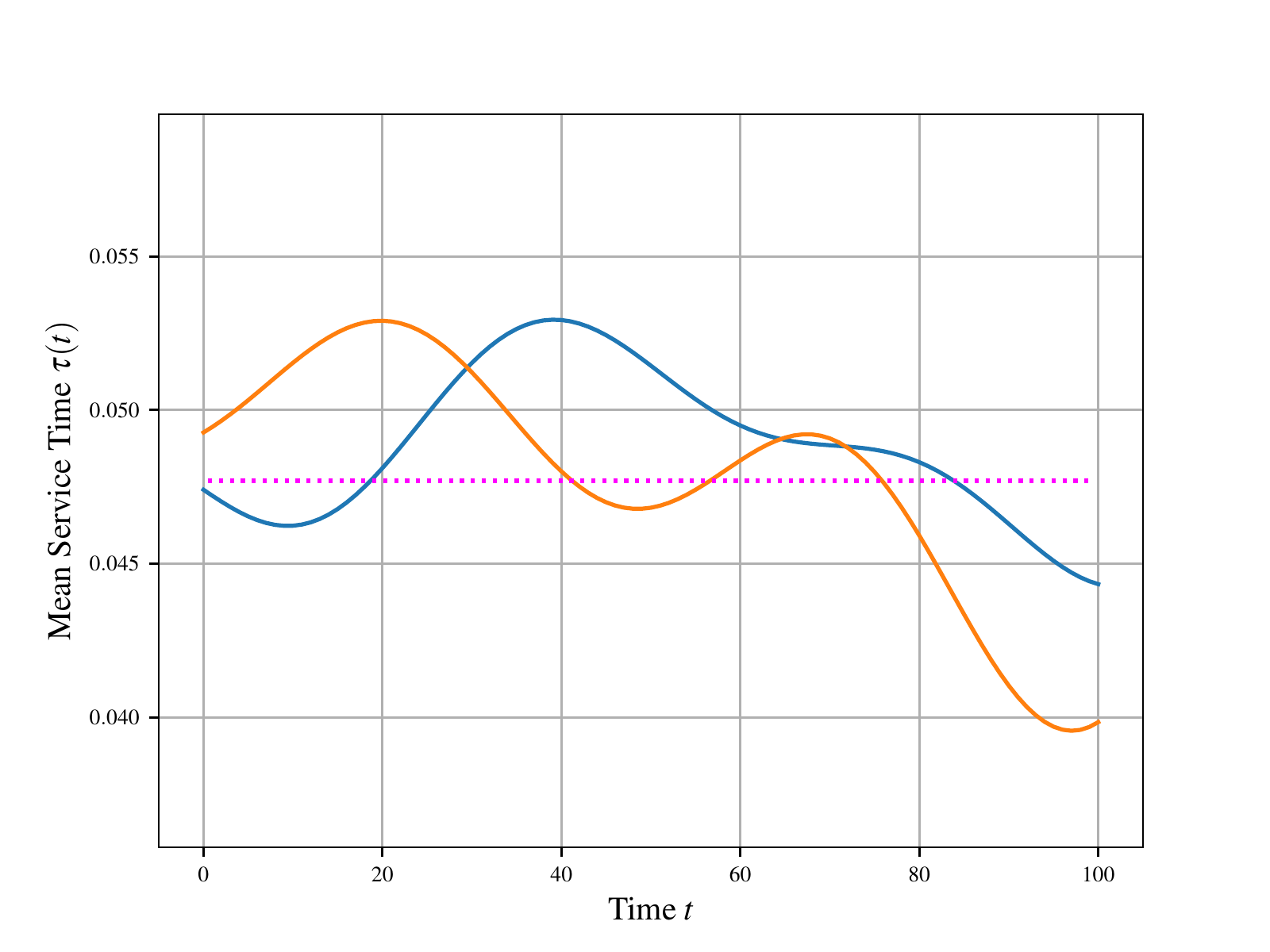}
        \caption{$\epsilon=0.2$}
        \label{fig:uncertain_arrivals:0.2}
    \end{subfigure}
    \caption[Examples of $\tau(t)$ realizations]{Examples of realizations of $\tau(t)$ for different values of $\epsilon$. Dotted lines indicate value of $\tau$.}
    \label{fig:uncertain_arrival_times}
\end{figure}

%
%
\begin{table}[bt]
    \centering
    \begin{subtable}{0.49\columnwidth}
        \centering
        \begin{tabular}{cc}
        \toprule
        $\epsilon$ & Improvement $\Delta_{12}$ (\%) \\
        \midrule
              0.01 &             1.42 \\
              0.02 &             2.93 \\
              0.05 &             6.76 \\
              0.10 &            11.76 \\
              0.20 &            19.56 \\
        \bottomrule
        \end{tabular}
        \caption{For 10 servers, 100 flows}
    \end{subtable}
    \begin{subtable}{0.49\columnwidth}
        \centering
        \begin{tabular}{cc}
        \toprule
        $\epsilon$ & Improvement $\Delta_{12}$ (\%)  \\
        \midrule
          0.01 &             1.47 \\
          0.02 &             2.89 \\
          0.05 &             6.91 \\
          0.10 &            12.04 \\        
          0.20 &            19.29 \\
        \bottomrule
        \end{tabular}
        \caption{For 20 servers, 200 flows}
    \end{subtable}
    \caption[Mean relative improvement of \gls{modelB}]{Mean relative percent improvement of the \gls{modelB} model over the \gls{modelA}.}
    \label{tab:results:mean_improvement}
\end{table}

%
%

\section{Discussion and Conclusion}
\label{sec:conclusion}

\subsection{Discussion}
In this research we have elicited a new robust model for Fluid Processing Network problems with control over
the proportion of server-effort dedicated to serving different flows.
Our primary contribution is to prove that this model is better than the transformed  \gls{modelA}
of \cite{bertsimas2014robust} for networks whose control is the proportion of server effort on each class.

Our second contribution is extending both models to cover box, budgeted, 
one-sided budgeted and polyhedral uncertainty sets.
These shapes have the advantage that they are less conservative than box uncertainty.
Robust optimization problems with budgeted uncertainty occur frequently, 
since there may be interdependence between parameter values.
Uncertainty set modeling may be based on observed values of the uncertain parameters.
In this case,
one may decide to model the uncertainty set as the convex hull of these observations,
which corresponds to polyhedral uncertainty.

Finally,
our third contribution is to quantitatively demonstrate, using numerical experiments, how our new model outperforms the current \gls{modelA} of \cite{bertsimas2014robust} for server-effort controls.
In our experiments, the relative improvement of our model increased with the magnitude of the uncertainty,
although we did not see any difference for models of different sizes.

\subsection{Limitations}
\label{subsec:limitations}
While our model improves on the \gls{modelA} when control is over server effort,
it does not generalize to standard fluid processing networks,
where 
the rates of flow are under control so the model of
\cite{bertsimas2014robust} is more natural.
Our computational results covered a limited subset of possible fluid processing networks,
in particular those for which the actual controls are server-effort proportion.
Next, while we were able to formulate tractable robust counterparts for several classes of uncertainty sets,
there are other useful types of uncertainty sets which we did not handle.
Bertsimas, Brown and Caramanis \cite{Bertsimas_Brown_Caramanis_2011} describe ellipsoidal,
cardinality constrained, and general norm uncertainty sets, for which we know of no robust formulation of SCLP.
Lorca and Sun \cite{Lorca_Sun_2015} describe dynamic uncertainty sets which are not supported by the SCLP-simplex algorithm which currently supports only a constant matrix.
Finally, our computational results covered only box uncertainty sets.
This is due to a limitation in the solver \cite{Shindin_SCLP-Python_2021}
in handling degenerate problems.

\subsection{Future work}
In this paper,
we assume that uncertainty set of our robust model is known.
We would like to study the statistical properties of the network state and controls by considering the stochastic nature of the uncertain parameters estimation.
The solver \cite{Shindin_SCLP-Python_2021} used in this work does not handle degenerate problems.
We would like to extend this solver to handle degenerate problems to solve robust counterparts of problems with additional shapes of uncertainty sets.
In this paper,
we solved robust SCLP problems with constant arrival and service rates.
We plan to extend the theory of SCLP and the solver \cite{Shindin_SCLP-Python_2021}
to include piece-wise constant rates. 

\bibliography{references}
\bibliographystyle{IEEEtran}

%
%
\appendix

\section{Proofs}
\label{appendix:proofs}

\begin{proof}[\Cref{prop:balance-model-a}]
    \label{prf:balance-model-a}
    For all mentioned uncertainty shapes, we note that each individual system dynamics constraint contains only single uncertain parameter,
    so all uncertainty sets could be considered as box uncertainty related to this parameter.
    For any buffer $k$ and for every realization of $\xi$ ,
    the balance constraint (\ref{eqn:dyn1g})
    is:
    \begin{equation}
        \label{eqn:balance-bertsimas}
            \alpha_k + \left( \oo{\lambda}_k + \ot{\lambda}_k\xi_k(t) \right) t
                -  \sum_{j} \!\! \int_0^t G_{kj} u_{j}(s) \dd s 
                \geq 0 ~ \forall t \text{.}            
    \end{equation}

    We are going to express \Cref{eqn:balance-bertsimas} in terms of fixed parameters,
    while including all cases of uncertainty.
    For box uncertainty,
    $|\xi_k(t)| \le 1$ means that
    $\oo{\lambda}_k + \ot{\lambda}_k\xi_k(t) \ge \oo{\lambda}_k - \ot{\lambda}_k ~\forall t$.
    So the reformulated constraint
    \begin{equation*}
               \alpha_k + \left( \oo{\lambda}_k - \ot{\lambda}_k \right) t
                               +  \sum_{j} \int_0^t G_{kj} u_{j}(s) \dd s 
                \geq 0 ~ \forall t, k  \text{.} 
    \end{equation*}
    covers all realizations of uncertainty.
\end{proof}

\begin{proof}[\Cref{prop:server-model-a}]
    \label{prf:server-model-a}
    For box uncertainty,
    $|\zeta_j(t)| \le 1 ~ \forall t, i$ implies that
    $\left( \oo{\tau}_j + \ot{\tau}_j \zeta_j(t) \right) \le
        \left( \oo{\tau}_j + \ot{\tau}_j \right) ~\forall t, j$
    and hence the robust counterparts of \Cref{eqn:eff1g} is: 
    $$\sum\limits_{j: s(j)=i} \left( \oo{\tau}_j + \ot{\tau}_j \right) u_j(t) \le 1 $$ 

    The proof for the one-sided budgeted uncertainty can be found in \cite[Theorem 1]{bertsimas2014robust}. For the other types of uncertainty,
    we also follow the method of proof of \cite[Theorem 1]{bertsimas2014robust}.
    The uncertain server capacity constraints (\ref{eqn:eff1g}),
    must hold for all realizations of $\zeta(t)$ within uncertainty set.
    In particular,
    for any $\zeta(t)$ that maximizes LHS of these constraints. This introduces 
    \gls{lp} problems,
    for each server $i$ and for each $t$.

    For budgeted uncertainty set it can be expressed by:
    \begin{eqnarray}
    \label{prob.budgetA}
        \max\limits_{\zeta(t)} & \sum\limits_{j: s(j)=i} \left( \oo{\tau}_j + \zeta_j(t) \ot{\tau}_j \right) u_j(t) \nonumber \\
        s.t. & \sum\limits_{j: s(j)=i}  |\zeta_j(t)| \leq \Gamma_i, \quad  |\zeta_j(t)| \leq 1.
    \end{eqnarray}
    
    To find the dual of (\ref{prob.budgetA}),
    we note that $\sum |\zeta_j(t)| \leq \Gamma_i$ is equivalent to the system
    $-w_j(t) \leq \zeta_j(t) \leq w_j(t)$,
    $\sum w_j(t) \leq \Gamma_i$,
    $w_j(t) \leq 1$,
    $w_j(t) \geq 0$.
    With this,
    we can consider our problem a lower-dimensional representation of a \gls{lp} with added variables $w_j, j=1,\ldots,J$
    \cite[Section 1.3]{BenTalElGhaouiNemirovski+2009}.
    Considering Lagrange multipliers
    $\beta_i(t)$, $\gamma_j(t)$, $\phi_j(t)$, $\nu_j(t)$
    we derive the dual:
    \begin{eqnarray}
        \min\limits_{\beta(t),\gamma(t)} &\quad \Gamma_i \beta_i(t) + \sum\limits_{j: s(j)=i} \left(\oo{\tau}_j u_j(t) + \gamma_j(t) \right)\nonumber \\
        s.t. &\quad -\phi_j(t) - \nu_j(t) + \beta_i(t) + \gamma_j(t)  \geq 0,\nonumber \\
        &\quad \phi_j(t) - \nu_j(t) =  \ot{\tau}_j(t) u_j(t), \nonumber\\
            &\quad \beta_i(t), \gamma_j(t), \phi_j(t), \nu_j(t)
                \geq 0, ~\text{for}~ s(j)=i\text{.}\nonumber
    \end{eqnarray}
     Furthermore, since $\ot{\tau}_j(t) u_j(t) \ge 0$, we can set $\nu_j(t) = 0$ and hence problem could be reduced to:
    \begin{eqnarray}
        \min\limits_{\beta(t),\gamma(t)} &\quad  \Gamma_i \beta_i(t) + \sum\limits_{j: s(j)=i} \left(\oo{\tau}_j u_j(t) + \gamma_j(t) \right) \nonumber \\
        s.t. &\quad  \beta_i(t) + \gamma_j(t)  \geq  \ot{\tau}_j(t) u_j(t), \nonumber \\
                   &\quad \beta_i(t), \gamma_j(t)
                \geq 0, ~\text{for}~ s(j)=i\text{.} \nonumber 
   \end{eqnarray}
    
    For polyhedral uncertainty $D\zeta(t) + d \geq 0$ where $D \in \R^{M \times J}$
    and $d \in \R^{M}$,
    we similarly derive the following dual problem from which the constraints in
    \Cref{tab:robust-counterpart-equation-server-capacity-a} are derived:
    \begin{eqnarray}
        \min\limits_{\delta(t)} &\quad - \sum\limits_{m} d_{m} \delta_m(t) \nonumber \\
        s.t. &\quad \sum\limits_{m} D_{mj} \delta_{m}(t) =   u_{j}(t) \ot{\tau}_{j}, \quad \delta(t) \geq 0 \text{.} \nonumber 
    \end{eqnarray}
    
    These dual problems are equivalent to the constraints in \Cref{tab:robust-counterpart-equation-server-capacity-a}.
\end{proof}

\begin{proof}[\Cref{prop:balance-model-b}]
    \label{prf:balance-model-b}
    For all realizations of uncertainty and all $t, k, j$,
    we have $\oo{\lambda}_k - \ot{\lambda}_k \le \oo{\lambda}_k + \ot{\lambda}_k \xi_k(t) \le \oo{\lambda}_k + \ot{\lambda}_k$,
    and $\oo{\mu}_j - \ot{\mu}_j \le \oo{\mu}_j + \ot{\mu}_j \zeta_j(t) \le \oo{\mu}_j + \ot{\mu}_j$.
    Our constraint must hold for all values of uncertain parameters,
    in particular those that minimize the buffer quantity.
    The result for box uncertainty is therefore clear.
    
    For the other types of uncertainty,
    let $S$ be the compact set of possible values of $\zeta(s)$.
    Then we can exchange the order of integration and maximization.
    Furthermore,
    $\max\limits_{s \in \mathcal{S}} \zeta(s) = \max \zeta(t)$
    and since $G_{kj}$ and $\ot{\mu}_j$ are constants,
    and $\eta_j(s) \geq 0$,
    we can extract $\zeta_j(s)$ from inside the integral to $\zeta_j(t)$
    outside,
    and we our constraint is now a linear function of $\zeta_j(t), j=1,\ldots,J$.
    
    For one-sided budgeted uncertainty,
    we note that for inflows,
    the minimum occurs when $\zeta_j(t)=0$.
    Taking the maximum to be the negative of the minimum,
    we derive a separate \gls{lp} subproblem for each server $i$:
    \begin{eqnarray}
        \max\limits_{\zeta(t)} \quad & - \sum\limits_{j:k(j) \ne k}  \zeta_{j}(t)
            \int_0^t G_{kj} \ot{\mu}_{j} \eta_{j}(s) \dd s \nonumber \\
        s.t. \quad 
            & \sum\limits_{j:s(j) = i} \zeta_j(t) \leq \Gamma_i,  \quad 0 \leq \zeta_j(t) \leq 1.\nonumber 
    \end{eqnarray}
    Which has symmetric dual:
    \begin{eqnarray}
       \min_{\beta(t), \gamma(t)} ~ & \sum_i \Gamma_i \beta_{ki}(t)
            + \sum\limits_{i} \sum\limits_{j:s(j)=i} \gamma_{kij}(t) \nonumber  \\
        s.t. ~ & \beta_{ki}(t) + \gamma_{kij}(t) 
            \geq - G_{kj} \ot{\mu}_j
                \int\limits_0^t \eta_j(s) \dd s ~\forall i, j:s(j)=i \nonumber  \\
            & \beta(t), \gamma(t) \geq 0 \nonumber 
    \end{eqnarray}
    For budgeted uncertainty,
    our proof
    uses the same procedure of adding extra variables $w_j(t)$ as the proof of \Cref{prop:server-model-a},
    and creates separate \gls{lp} subproblems for each server as we did for
    one-sides.
    The duals of these have the form:
    \begin{eqnarray}
         \min_{\beta(t),\gamma(t),\delta(t), \nu(t)} \quad & \sum\limits_{i} \Gamma_i \beta_{ki}(t)
            + \sum\limits_{i} \sum\limits_{j:s(j)=i} \gamma_{kij}(t) \nonumber \\
        s.t. \quad & \delta_{ki}(t) - \nu_{kij}(t) = -G_{kj} \ot{\mu}_j
               \int\limits_0^t  \eta_j(s) \dd s ~\forall i, j:s(j)=i \nonumber \\
            & \beta_{ki}(t) + \gamma_{kij}(t) - \delta_{kij}(t) - \nu_{kij}(t) \geq 0 \nonumber \\
            & \beta_{ki}(t), \gamma_{kij}(t), \delta_{kij}(t), \nu_{kij}(t) \geq 0 \nonumber 
    \end{eqnarray}
    From which we combine the two constraints using the substitution $\nu_j(t) = \delta_j(t) + G_{kj} \ot{\mu}_j \int_0^t \eta_j(s)\dd s$.
    This gives the result in \Cref{tab:robust-counterpart-equation-balance-b}.
    
    For polyhedral uncertainty sets,
    the single corresponding dual problem is:
   \begin{eqnarray}
        \min_{\delta(t)} ~ & \sum\limits_{m} d_{m} \delta_{m}(t) \nonumber \\
        s.t. ~ & \sum\limits_{m} D_{mj} \delta_{m} 
            = - G_{kj} \ot{\mu}_j \int\limits_0^t \eta_{j}(s) \dd s, \quad \delta(t) \geq 0 \text{.} \nonumber 
    \end{eqnarray}
    This gives the result in \Cref{tab:robust-counterpart-equation-balance-b}.

\end{proof}




\end{document}